\documentclass[11pt]{amsart}
\usepackage{geometry}
\usepackage{amsmath, amssymb, amsthm, mathrsfs}
\usepackage{mathtools}
\usepackage[protrusion=true,expansion=false]{microtype}
\usepackage{enumitem}
\usepackage[dvipsnames]{xcolor}
\usepackage{hyperref}
\hypersetup{colorlinks=true, citecolor=red, linkcolor=blue, filecolor=magenta, urlcolor=red}
\usepackage[capitalize,nameinlink]{cleveref}

\theoremstyle{plain}
\newtheorem{thm}{Theorem}[section]

\newtheorem{lem}[thm]{Lemma}
\newtheorem{prop}[thm]{Proposition}

\theoremstyle{definition}
\newtheorem{defn}[thm]{Definition}
\newtheorem{ques}[thm]{Question}
\newtheorem{rem}[thm]{Remark}

\crefname{thm}{Theorem}{Theorems}
\crefname{lem}{Lemma}{Lemmas}
\crefname{prop}{Proposition}{Propositions}
\crefname{cor}{Corollary}{Corollaries}
\crefname{defn}{Definition}{Definitions}

\newcommand{\fp}{\mathfrak{p}}
\newcommand{\fq}{\mathfrak{q}}
\newcommand{\fm}{\mathfrak{m}}
\newcommand{\fn}{\mathfrak{n}}
\newcommand{\bbC}{\mathbb{C}}
\newcommand{\bbQ}{\mathbb{Q}}
\newcommand{\bbZ}{\mathbb{Z}}
\newcommand{\bbR}{\mathbb{R}}
\newcommand{\bbP}{\mathbb{P}}

\newcommand{\bbA}{\mathbb{A}}
\newcommand{\sO}{\mathcal{O}}

\DeclareMathOperator{\Spec}{Spec}
\DeclareMathOperator{\Frac}{Frac}
\DeclareMathOperator{\Cl}{Cl}
\DeclareMathOperator{\Pic}{Pic}

\DeclareMathOperator{\depth}{depth}
\DeclareMathOperator{\link}{link}
\DeclareMathOperator{\tci}{tci}
\DeclareMathOperator{\Hom}{Hom}
\DeclareMathOperator{\trdeg}{trdeg}
\DeclareMathOperator{\lcm}{lcm}

\DeclareMathOperator{\charac}{char}
\DeclareMathOperator{\ord}{ord}
\DeclareMathOperator{\divisor}{div}

\title{Boundedness of total Cartier indices for rational singularities in families}
\author{Jihao Liu, Ruicheng Hu, Sheng Qin}
\address{Department of Mathematics, Peking University, No. 5 Yiheyuan Road, Haidian District, Beijing 100871, China}
\address{Beijing International Center for Mathematical Research, Peking University, No. 5 Yiheyuan Road, Haidian District, Beijing 100871, China}
\email{liujihao@math.pku.edu.cn}
\email{huruicheng@stu.pku.edu.cn}
\email{qinsheng@stu.pku.edu.cn}
\date{\today}

\begin{document}

\begin{abstract}
We show that the total Cartier index of varieties with rational singularities in a bounded family is bounded. This solves a problem of Han and Jiang. The overall structure of the proof, which treats the surface case and the higher-dimensional case separately, was originated by generative AI, particularly the Rethlas system, and was substantially corrected and elaborated by hand.
\end{abstract}

\maketitle

\tableofcontents

\section{Introduction}\label{sec:intro}

We work over an algebraically closed field $k$ of characteristic $0$.

A result of Greb--Kebekus--Peternell \cite[Theorem~1.10 and Remark~1.11]{GKP16} shows that, for any variety of klt type \(X\), there exists a positive integer \(I\) such that \(ID\) is Cartier for every \(\mathbb Q\)-Cartier Weil divisor \(D\) on \(X\). Such a definition can be generalized beyond the category of klt type varieties. More precisely, we may define the following.

\begin{defn}[Total Cartier index]\label{def:tci}
Let $X$ be a normal noetherian scheme. For any $\mathbb Q$-Cartier Weil divisor $D$ on $X$, the \emph{Cartier index} of $D$ is the minimal positive integer $I$ such that $ID$ is Cartier. The \emph{total Cartier index} of $X$ is
\[
\tci(X) \coloneqq \lcm\{ \text{Cartier indices of all $\bbQ$-Cartier Weil divisors on } X \},
\]
when this least common multiple is finite, and $+\infty$ otherwise.
\end{defn}
A more general question to ask is whether in a bounded family of varieties $X\rightarrow B$, we have that $\tci(X_b)$ is bounded for any fiber. This is of course too optimistic as it is possible that $\tci(X_b)=+\infty$ for some $b$ by considering a cone over an elliptic curve. Nevertheless, \cite[Theorem~1.2]{HJ26} proved that, for a bounded family of varieties, the total Cartier indices of the fibers that are of klt type are bounded. With this in mind, Han and Jiang asked the following:

\begin{ques}[{\cite[Problem~4.5]{HJ26}}]\label{ques:HJ26-4.5}
   Are the total Cartier indices bounded for a bounded family of varieties with rational singularities?
\end{ques}

The goal of this paper is to provide an affirmative answer to Question~\ref{ques:HJ26-4.5}:

\begin{thm}\label{thm:main}
Let $\pi: X\rightarrow B$ be a projective morphism to a finite type $k$-scheme. Then there exists a positive integer $I$ satisfying the following. For any closed point $b\in B$ such that $X_b$ is normal projective of pure dimension with at most rational singularities, $\tci(X_b)$ divides $I$.
\end{thm}

\subsection*{Strategy of the proof}

We remark that \cite[Theorem 1.2]{HJ26} proved the case for klt type fibers, which is a weaker condition than rational singularities. The proof of \cite[Theorem 1.2]{HJ26} heavily relies on the minimal model program in families. In this paper, we choose an alternative approach.

The proof has a local core and a global reduction. The local core is the algebra theorem, namely Theorem~\ref{thm:universal_local}: if a normal rational local ring is as a prime localization of a complex affine scheme cut out in \(\bbA^N\) by a bounded number of equations of bounded degree, then the Cartier index of every closed point is bounded in terms only of those numerical bounds.

The proof of Theorem~\ref{thm:universal_local} is by induction on the local dimension. Dimension $\leq 1$ is immediate. In every higher dimension, the induction first provides a uniform multiple of the divisor which is Cartier on the punctured spectrum. The remaining index at closed points is then controlled by the torsion in \(H^2\) of the topological link. For surfaces, we use Brieskorn's computation of the class group of rational complex surface singularities and Poincar\'e duality for the link \cite{Bri68}. In dimension $\geq 3$, we use Koll\'ar's exponential sequence argument to inject the punctured local Picard group into \(H^2\) of the link \cite{Kol16}. A semialgebraic boundedness theorem, which comes from real algebraic geometry \cite{BPR06}, gives a uniform bound for this torsion in bounded degree.

It remains to pass from closed complex points to the prime localizations which occur in the induction. This is done by a finite residue field Noether normalization argument. After adjoining an algebraic closure of the residue field, the prime localization becomes a faithfully flat local base change of a closed point local ring over an algebraically closed field isomorphic to \(\bbC\). Flat divisorial base change transports the divisor to this closed point model, and faithfully flat descent then brings the resulting Cartierness back to the original local ring.

The global theorem follows by covering a projective family by finitely many affine charts of bounded degree, applying the universal local theorem on each chart, and reducing global \(\tci\) to closed point local Cartier indices. The final passage from complex coefficients to an arbitrary algebraically closed field of characteristic zero is by spreading out the family, the point, and the divisor over a finitely generated field, embedding that field into \(\bbC\), applying the complex result, and descending Cartierness by faithfully flat local base change.

\begin{rem}
Although out of the topic of this paper, it is worthy to mention that since our approach does not rely on the minimal model program, a lot of questions in moduli theory of structures with rational singularities (including some already solved questions in KSBA moduli and K-moduli theory) might be able to be solved in an easier and more straightforward way. We leave the applications to readers of potential interest. 
\end{rem}

\subsection*{Organization}

Section~\ref{sec:prelim} recalls some preliminaries and algebraic tools that will be used in the rest of the paper. Section~\ref{sec:layer4d} contains the topological input from semialgebraic topology, surface rational singularities, and higher-dimensional local Picard groups. Section~\ref{sec:local_estimates} proves the bounded degree local residual estimates, including the Noether normalization reduction from arbitrary prime localizations to closed points. Section~\ref{sec:universal_local} proves the universal local theorem. Section~\ref{sec:layer012} globalizes the local theorem, makes field descent, and proves Theorem~\ref{thm:main}.

\begin{rem}
The overall structure of the proof of this paper --- in particular the strategy of treating the surface case and the higher-dimensional case separately and assembling them into the local-to-global argument --- was originated by generative AI, particularly the Rethlas system, by putting together known ingredients. More precisely, the first prompt was generated by ChatGPT Pro 5.5, and the rest was handled by Rethlas. The first draft obtained in this way contained many errors. Most of them were minor and readily fixable by hand; however, the proof of Lemma~\ref{lem:semialg_torsion} was essentially wrong in the AI-generated draft and was corrected only later. A substantial amount of elaboration, verification, and rewriting was then carried out by hand to reach the present version. 

See \cite{Ju+26} for a detailed introduction to the Rethlas system. Due to the limitation of generative AI, it is possible that we have missed some related references in the literature, and we welcome any comments from experts.
\end{rem}

\subsection*{Acknowledgements}
The first author was partially supported by the National Key R\&D Program of China \#\allowbreak 2024YFA1014400. The authors would like to thank the Rethlas team, namely Haocheng Ju, Jiedong Jiang, Shurui Liu, Guoxiong Gao, Yuefeng Wang, Zeming Sun, Bin Wu, Liang Xiao, and Bin Dong, for their contributions to the development of Rethlas and its customized version used for the problem studied in this paper. The authors would like to thank Kaiyuan Gu for assistance with the verification of an earlier blueprint of this paper. The authors would like to thank Ruochuan Liu and Gang Tian for constant support and encouragement.

\section{Algebraic preliminaries and descent}\label{sec:prelim}

We adopt the standard notation and terminology for the minimal model program from~\cite{Sho92,KM98,BCHM10} and use them freely. In particular we refer the reader to \cite[Chapter 5]{KM98} for basic definitions and properties of rational singularities.

\begin{defn}[Divisorial module]\label{def:div_module}
Let $R$ be a normal noetherian domain with fraction field $K$. Let $D=\sum_{P\in E} a_P P$ be a Weil divisor on $\Spec R$, with finite support $E$ contained in the set of height one primes of $R$. The \emph{divisorial module} $\sO_R(D)$ is the fractional $R$-submodule of $K$ given by
\begin{align*}
\sO_R(D) &= \{ f\in K^\times \mid \divisor_R(f) + D \geq 0 \} \cup \{0\} \\
&= \bigcap_{P\in E} \pi_P^{-a_P} R_P \;\cap\; \bigcap_{P\notin E,\,\mathrm{ht}(P)=1} R_P,
\end{align*}
where $\pi_P$ is a uniformizer of $R_P$. It is a finite reflexive $R$-module of rank one. We write $\sO_R$ for $\sO_R(0)=R$. A Weil divisor $D$ on $\Spec R$ is Cartier at $\fp\in \Spec R$ if and only if $\sO_R(D)_\fp$ is a free $R_\fp$-module of rank one.
\end{defn}

\begin{defn}[Topological link]\label{def:link}
Let $x$ be a closed point of a complex algebraic variety $X$. Choose a closed embedding $X\hookrightarrow \bbC^M$ sending $x$ to the origin. For a sufficiently small positive real number $\epsilon$, the intersection of the underlying complex analytic set of $X$ with the Euclidean sphere of radius $\epsilon$ centered at the origin in $\bbC^M = \bbR^{2M}$ is a closed bounded real semialgebraic subset of $\bbR^{2M}$ whose homeomorphism type is independent of the choices of embedding and $\epsilon$ \cite{Har80,Ver76}. This homeomorphism type is the \emph{topological link} of the complex analytic germ $(x,X)$, denoted by $\link(x,X)$.
\end{defn}

\subsection{Flat base change and divisorial pullback}\label{sec:layer4a}

In this subsection, we prove the flat base change statements for reflexive and divisorial modules that will be used in the descent arguments.

\subsubsection{Flat preservation of reflexivity}

We begin with the elementary reflexivity statement for finite modules under flat base change.

\begin{lem}[Flat base change preserves reflexivity]\label{lem:flat_reflexive}
Let $R$ be a noetherian ring. Let $S$ be a noetherian flat $R$-algebra. Let $M$ be a finite $R$-module. Assume that $M$ is a reflexive $R$-module.
Then $S\otimes_R M$
is a reflexive $S$-module.
\end{lem}

\begin{proof}
Set $M^\vee \coloneqq \Hom_R(M,R)$. Since $R$ is noetherian and $M$ is finite, $M^\vee$ is a finite $R$-module.

Choose a finite presentation
\[
R^u \to R^v \to M \to 0
\]
with integers $u, v\geq 0$. Applying $\Hom_R(-,R)$ gives the exact sequence
\[
0 \to M^\vee \to R^v \to R^u.
\]
Tensor this exact sequence with $S$. Since $S$ is flat over $R$, the result is the exact sequence
\[
0 \to S\otimes_R M^\vee \to S^v \to S^u.
\]
The original presentation, after tensoring with $S$, is the finite presentation
\[
S^u \to S^v \to S\otimes_R M \to 0.
\]
Applying $\Hom_S(-,S)$ to this presentation gives the exact sequence
\[
0 \to \Hom_S(S\otimes_R M, S) \to S^v \to S^u.
\]
The two displayed kernels inside $S^v$ are the kernel of the same map $S^v\to S^u$. Thus
\[
\Hom_S(S\otimes_R M, S) \cong S\otimes_R M^\vee.
\]
Apply the same argument to the finite $R$-module $M^\vee$. It gives
\[
\Hom_S(S\otimes_R M^\vee, S) \cong S\otimes_R \Hom_R(M^\vee, R).
\]
By hypothesis, the natural map $M\to \Hom_R(M^\vee, R)$ is an isomorphism. Tensoring it with $S$ gives an isomorphism
\[
S\otimes_R M \to S\otimes_R \Hom_R(M^\vee, R).
\]
Combining with the preceding identification of $\Hom_S(S\otimes_R M^\vee, S)$, this is exactly the natural double dual map
\[
S\otimes_R M \to \Hom_S(\Hom_S(S\otimes_R M, S), S).
\]
Thus that natural map is an isomorphism.
\end{proof}

\subsubsection{Divisorial base change}

The following form of flat divisorial base change is the only one needed later. It includes the case where a codimension one point of the target lies over the generic point of the source.

\begin{lem}[Flat divisorial base change]\label{lem:generalized_flat_pullback}
Let $R$ and $S$ be normal noetherian domains. Let
\[
\phi\colon R\to S
\]
be an injective flat ring homomorphism. Let
\[
D=\sum_{P\in E}a_P P
\]
be a Weil divisor on $\Spec R$, where $E$ is a finite set of height one primes of $R$ and $a_P\in\bbZ$.

For every height one prime $Q$ of $S$, put $P_Q\coloneqq \phi^{-1}(Q)$. Then $P_Q$ is either the zero ideal or a height one prime of $R$. Define an integer $b_Q$ as follows. If $P_Q=0$, set $b_Q=0$. If $P_Q$ has height one, let $e_Q$ be the valuation in the discrete valuation ring $S_Q$ of the image under $\phi$ of a generator of the principal ideal $P_QR_{P_Q}$, and set
\[
b_Q\coloneqq a_{P_Q}e_Q,
\]
where $a_{P_Q}=0$ if $P_Q\notin E$. Then only finitely many $b_Q$ are nonzero. The Weil divisor
\[
D_S\coloneqq \sum_Q b_QQ
\]
on $\Spec S$ is the flat divisorial base change of $D$ to $\Spec S$.

For every integer $q\geq 1$, the flat divisorial base change of $qD$ is $qD_S$, and there is an isomorphism of $S$-modules
\[
S\otimes_R \sO_R(qD)\cong \sO_S(qD_S).
\]
\end{lem}

\begin{proof}
Let $Q$ be a height one prime of $S$, and put $P\coloneqq \phi^{-1}(Q)$. Since $\phi$ is flat, going-down for flat ring maps shows that $P$ has height at most one. As $R$ is a domain, $P$ is either zero or a height one prime. This proves the contraction assertion.

We next prove the finiteness of the support of $D_S$. If $b_Q\neq 0$, then $P_Q$ belongs to the finite set $E$. Fix $P\in E$. We claim that every height one prime $Q$ of $S$ with $Q\cap R=P$ is a minimal prime of $PS$. Let $Q_0\subseteq Q$ be a minimal prime of $S$ over $PS$. Then $Q_0\cap R$ contains $P$ and is contained in $Q\cap R=P$, hence $Q_0\cap R=P$. Choose a nonzero element $r\in P$. Since $\phi$ is injective, the image of $r$ in $S$ is nonzero; since $r\in PS$, it lies in $Q_0$. Thus $Q_0$ is a nonzero prime of the domain $S$. As $Q_0\subseteq Q$ and $Q$ has height one, we get $Q_0=Q$. The noetherian ring $S$ has only finitely many minimal primes over $PS$. Since $E$ is finite, only finitely many $Q$ have $b_Q\neq 0$.

We now verify that the displayed coefficient formula is the flat pullback of \(D\) at each codimension one point of \(\Spec S\). Let \(Q\) be a height one prime of \(S\), and set \(P=Q\cap R\). If \(P=0\), then no height one component of \(D\) specializes to \(Q\), and the pullback coefficient at \(Q\) is zero. If \(P\) has height one, then \(R_P\) and \(S_Q\) are discrete valuation rings. A generator of the principal ideal \(P R_P\) has valuation \(1\) in \(R_P\), and its image in \(S_Q\) has valuation \(e_Q\). Therefore a component \(a_P P\) pulls back with coefficient \(a_Pe_Q\) at \(Q\). This is exactly the definition of \(b_Q\). The same valuation computation is linear in the coefficients, so the base change of \(qD\) is \(qD_S\).

It remains to identify the divisorial modules. We prove the assertion for a fixed $q\geq 1$. Let $K\coloneqq \Frac(R)$ and $L\coloneqq \Frac(S)$. The injective map $\phi$ identifies $K$ with a subfield of $L$. Since $\sO_R(qD)$ is a fractional $R$-submodule of $K$ and $S$ is flat over $R$, the module $S\otimes_R\sO_R(qD)$ may be regarded as the fractional $S$-submodule of $L$ generated by the image of $\sO_R(qD)$.

Localize at a height one prime $Q$ of $S$, and set $P\coloneqq Q\cap R$.

If $P=0$, then every nonzero element of $R$ maps to an element of $S\setminus Q$, hence to a unit of the discrete valuation ring $S_Q$. Since $\sO_R(qD)\otimes_RK=K$, we obtain
\[
(S\otimes_R\sO_R(qD))_Q=S_Q.
\]
The coefficient of $Q$ in $qD_S$ is zero, so $\sO_S(qD_S)_Q=S_Q$.

If $P$ has height one, then $R_P$ and $S_Q$ are discrete valuation rings. Let $\pi_P$ be a generator of $PR_P$. The coefficient of $P$ in $qD$ is $qa_P$, with $a_P=0$ if $P\notin E$. Therefore
\[
\sO_R(qD)_P=\pi_P^{-qa_P}R_P,
\]
and hence
\[
(S\otimes_R\sO_R(qD))_Q
=S_Q\otimes_{R_P}\pi_P^{-qa_P}R_P
=\phi(\pi_P)^{-qa_P}S_Q.
\]
By definition of $e_Q$, the element $\phi(\pi_P)$ has valuation $e_Q$ in $S_Q$. Thus the last displayed module equals $\sO_S(qD_S)_Q$.

We have shown that the two fractional $S$-modules $S\otimes_R\sO_R(qD)$ and $\sO_S(qD_S)$ have the same localization at every height one prime of $S$. The module $\sO_S(qD_S)$ is reflexive because $S$ is normal and noetherian. The module $\sO_R(qD)$ is finite and reflexive over $R$, so Lemma~\ref{lem:flat_reflexive} shows that $S\otimes_R\sO_R(qD)$ is reflexive over $S$. Over a normal noetherian domain, reflexive fractional rank one modules are determined by their localizations at the height one primes \cite[Chapter~VII, \S~4, Theorem~4]{Bou89}. Therefore
\[
S\otimes_R\sO_R(qD)=\sO_S(qD_S)
\]
as fractional $S$-submodules of $L$, which gives the asserted isomorphism.
\end{proof}

\subsection{Faithfully flat descent of Cartierness}\label{sec:layer4b}

In this subsection, we prove the faithfully flat descent statements that turn Cartier index bounds after base change into Cartier index bounds before base change.

\begin{lem}[Faithfully flat local base change detects Cartierness]\label{lem:ff_detects_cartier}
Let $R\to S$ be a faithfully flat local homomorphism of noetherian local rings. Let $M$ be a finite $R$-module. Assume that $S\otimes_R M$ is a free $S$-module of rank one. Then $M$ is a free $R$-module of rank one.

Consequently, let $R$ be a normal noetherian local domain, let $D$ be a Weil divisor on $\Spec R$, and let $q$ be a positive integer. If $S\otimes_R \sO_R(qD)$ is a free $S$-module of rank one, then $qD$ is Cartier on $\Spec R$.
\end{lem}

\begin{proof}
Let $\fm$ be the maximal ideal of $R$, and let $\fn$ be the maximal ideal of $S$. Since $R\to S$ is local, $\fm S \subset \fn$. Since $S\otimes_R M$ is a free $S$-module of rank one, the vector space
\[
(S\otimes_R M)/\fn(S\otimes_R M)
\]
has dimension one over $S/\fn$. The natural isomorphism
\[
(S/\fn)\otimes_{R/\fm}(M/\fm M) \cong (S\otimes_R M)/\fn(S\otimes_R M)
\]
therefore implies that $M/\fm M$ has dimension one over $R/\fm$.

Choose $u\in M$ whose image spans $M/\fm M$. By Nakayama's lemma, the $R$-linear map
\[
\alpha\colon R\to M, \qquad \alpha(1)=u
\]
is surjective. After tensoring with $S$, the map
\[
S\otimes_R \alpha\colon S \to S\otimes_R M
\]
is a surjection from $S$ onto a free rank one $S$-module, hence is multiplication by an element of that free rank one module that generates it, hence by a unit. Thus $S\otimes_R \alpha$ is an isomorphism.

Let $K\coloneqq \ker(\alpha)$. The exact sequence
\[
0 \to K \to R \xrightarrow{\alpha} M \to 0
\]
remains exact after tensoring with $S$, because $S$ is flat over $R$. Since $S\otimes_R \alpha$ is an isomorphism, $S\otimes_R K=0$. Since $S$ is faithfully flat over $R$, $K=0$. Thus $\alpha$ is an isomorphism, and $M$ is a free $R$-module of rank one.

For the consequence, the divisorial module $\sO_R(qD)$ is finite because $R$ is noetherian and normal. Apply the first part to $M=\sO_R(qD)$ to conclude that $\sO_R(qD)$ is free of rank one. For a Weil divisor on a normal local domain, this is equivalent to $qD$ being Cartier on $\Spec R$.
\end{proof}

\subsection{Field base change for rational singularities}\label{sec:layer4c}

This subsection records the behavior of rational singularities under field extension.

\begin{lem}[Field base change preserves normal rational local singularities]\label{lem:alg_field_rational}
Let $K$ be a field of characteristic $0$. Let $L$ be a field extension of $K$. Let $A$ be a finitely generated $K$-domain. Set $A_L\coloneqq L\otimes_K A$. Let $\fq$ be a prime ideal of $A_L$, and let $\fp$ be the contraction of $\fq$ to $A$. Set
\[
R\coloneqq A_\fp,
\qquad
S\coloneqq (A_L)_\fq.
\]
Assume that $R$ is a normal local domain and that $R$ has a rational singularity. Then $S$ is a normal local domain and $S$ has a rational singularity. Moreover
\[
\dim S=\dim R+\dim\bigl(S\otimes_R\kappa(\fp)\bigr).
\]
In particular, if $L/K$ is algebraic, then $\dim S=\dim R$.
\end{lem}

\begin{proof}
We first record the dimension computation. The rings $R$ and $S$ are noetherian local rings: indeed $A$ is a finitely generated algebra over the field $K$, and $A_L$ is a finitely generated algebra over the field $L$. The homomorphism
\[
R=A_\fp\longrightarrow (A_L)_\fq=S
\]
is local. It is flat because $K\to L$ is flat and flatness is preserved by base change and localization. Thus the hypotheses of the dimension formula for flat local homomorphisms of noetherian local rings are satisfied. Applying that formula gives
\[
\dim S=\dim R+
\dim\bigl(S/\fp S\bigr)
=
\dim R+
\dim\bigl(S\otimes_R\kappa(\fp)\bigr).
\]
If $L/K$ is algebraic, then $S\otimes_R\kappa(\fp)$ is a localization of the algebraic $\kappa(\fp)$-algebra $L\otimes_K\kappa(\fp)$. Such a ring has Krull dimension $0$. Hence, in the algebraic case,
\[
\dim S=\dim R.
\]

We next prove normality. Since $A$ is of finite type over a field, it is excellent; in particular its normal locus, the complement of $V(\text{conductor})$, is open. The point $\fp$ belongs to the normal locus because $A_\fp=R$ is normal. Hence there is an element $f\in A\setminus\fp$ such that $A_f$ is normal. Since $f\notin\fq$, the local ring $S$ is a localization of $L\otimes_K A_f$.

The field $K$ is perfect because $\charac(K)=0$ \cite[Tag~030Z]{Stacks}. A normal finite type algebra over a perfect field is geometrically normal \cite[Tag~037Z]{Stacks}. Therefore $L\otimes_K A_f$ is normal. Localizing at $\fq$ shows that $S$ is normal. Since $S$ is local and normal, $S$ is a domain.

It remains to prove that $S$ has a rational singularity. Put $X\coloneqq\Spec R$ and $Y\coloneqq\Spec S$. Since $R$ has a rational singularity, there is a proper birational morphism
\[
\mu\colon Z\to X
\]
with $Z$ regular and $R^i\mu_*\sO_Z=0$ for every $i>0$. Form the base change
\[
\mu_S\colon Z_S\coloneqq Z\times_XY\to Y.
\]
The morphism $\mu_S$ is proper. The map $R\to S$ is injective: multiplication by any nonzero element of the domain $R$ remains injective after the flat base change to $A_L$ and after the localization at $\fq$. Hence the generic point of $Y$ maps to the generic point of $X$. If $U\subset X$ is a dense open subset over which $\mu$ is an isomorphism, then $U\times_XY$ is a dense open subset of $Y$, and $\mu_S$ is an isomorphism over it. Thus $\mu_S$ is birational.

The scheme $Z_S$ is regular. Indeed, $Z$ is regular and essentially of finite type over the perfect field $K$; hence $Z$ is geometrically regular over $K$. Since every field extension in characteristic $0$ is separable, regularity is preserved after the field extension $K\subset L$ and after localization.

For every integer $i>0$, flat base change for higher direct images of quasi-coherent sheaves \cite[Cohomology of Schemes, Lemma~30.5.3, Tag~02KH]{Stacks} gives
\[
R^i(\mu_S)_*\sO_{Z_S}\cong S\otimes_R R^i\mu_*\sO_Z=0.
\]
Thus $\mu_S$ is a resolution of $Y$ with vanishing higher direct images of the structure sheaf. Therefore $S$ has a rational singularity.
\end{proof}

\section{Topological input for local Cartier indices}\label{sec:layer4d}

This section supplies the topological input used to control local Cartier indices. It begins with a uniform semialgebraic torsion bound, then treats rational surface singularities, and finally records the higher-dimensional Chern class input.

\subsection{Bounded semialgebraic cohomology torsion}

We need the following result from real algebraic geometry. The bound of the total Cartier indices in Theorem \ref{thm:main} essentially come from the bound $C(k,s,d)$ in the following lemma.

\begin{lem}[Bounded semialgebraic cohomology torsion]
\label{lem:semialg_torsion}
Let \(k,s,d\), and \(q\) be integers with \(k\ge 0\), \(s\ge 1\), \(d\ge 1\),
and \(q\ge 0\). There exists an integer \(M\ge 1\), depending only on
\(k,s\), and \(d\), with the following property.

Let \(\mathcal P\) be a finite set of at most \(s\) real polynomials in
\(\mathbb R[x_1,\ldots,x_k]\), each of total degree at most \(d\). Let
\(T\subset \mathbb R^k\) be a closed and bounded \(\mathcal P\)-semialgebraic
set, meaning that \(T\) is defined by a quantifier-free Boolean formula whose
atomic predicates are
\[
P(x)=0,\qquad P(x)>0,\qquad P(x)<0
\]
for polynomials \(P\in\mathcal P\). Then the torsion subgroup of
\(H^q(T,\mathbb Z)\) has order at most \(M\) and exponent at most \(M\).
\end{lem}

\begin{proof}
We first prove a uniform bound on the size of a triangulation. The algebraic zero set case is precisely the finite topological types theorem, namely \cite[Theorem~5.47]{BPR06}. We need the same conclusion for Boolean combinations of sign conditions, and we give the reduction to semialgebraic triviality \cite[Theorem~5.46]{BPR06}.

Let
\[
D_{k,d}:=\binom{k+d}{d}.
\]
A real polynomial in \(k\) variables of degree at most \(d\) is determined by
\(D_{k,d}\) coefficients. Thus ordered \(s\)-tuples of such polynomials are
parametrized by
\[
\mathbb R^N,\qquad N=sD_{k,d}.
\]
If fewer than \(s\) polynomials are used, we pad the list with zero
polynomials.

Write the universal polynomials as
\[
P_i(a;x)=\sum_{|\alpha|\le d} a_{i,\alpha}x^\alpha,
\qquad i=1,\ldots,s,
\]
where \(a=(a_{i,\alpha})\in \mathbb R^N\) is the coefficient parameter and
\(x=(x_1,\ldots,x_k)\in \mathbb R^k\).

For each sign vector
\[
\sigma=(\sigma_1,\ldots,\sigma_s)\in \{-1,0,+1\}^s,
\]
define
\[
S_\sigma
=
\left\{
(a,x)\in \mathbb R^N\times \mathbb R^k :
\operatorname{sign} P_i(a;x)=\sigma_i \text{ for all } i=1,\ldots,s
\right\}.
\]
Each \(S_\sigma\) is a semialgebraic subset of
\(\mathbb R^N\times \mathbb R^k\), and there are only finitely many such
subsets.

Consider the projection
\[
\pi:\mathbb R^N\times \mathbb R^k\longrightarrow \mathbb R^N.
\]
Apply \cite[Theorem 5.46]{BPR06} to \(\pi\) and to the finite family of
semialgebraic subsets \(S_\sigma\). We obtain a finite semialgebraic partition
\[
\mathbb R^N=\bigsqcup_{\alpha\in A} B_\alpha
\]
such that \(\pi\) is semialgebraically trivial over each \(B_\alpha\), and the
trivialization is compatible with every \(S_\sigma\). More explicitly, for
each \(\alpha\) there is a semialgebraic set \(F_\alpha\), semialgebraic
subsets \(F_{\alpha,\sigma}\subset F_\alpha\), and a semialgebraic
homeomorphism
\[
\theta_\alpha:B_\alpha\times F_\alpha
\xrightarrow{\sim}
\pi^{-1}(B_\alpha)
\]
such that \(\pi\circ\theta_\alpha\) is the projection
\(B_\alpha\times F_\alpha\to B_\alpha\), and
\[
\theta_\alpha(B_\alpha\times F_{\alpha,\sigma})
=
S_\sigma\cap \pi^{-1}(B_\alpha)
\]
for every sign vector \(\sigma\).

It follows that if \(a,a'\in B_\alpha\), then the induced homeomorphism between
fibers
\[
h_{a,a'}:
\{a\}\times\mathbb R^k
\xrightarrow{\sim}
\{a'\}\times\mathbb R^k
\]
sends the \(\sigma\)-sign condition set over \(a\) to the corresponding
\(\sigma\)-sign condition set over \(a'\), for every
\(\sigma\in\{-1,0,+1\}^s\).

Now a quantifier-free Boolean formula in the signs of \(P_1,\ldots,P_s\) is
equivalent to choosing a subset
\[
\Sigma\subset \{-1,0,+1\}^s
\]
of allowed sign vectors. The corresponding fiber over \(a\in\mathbb R^N\) is
\[
T_a^\Sigma
=
\bigcup_{\sigma\in\Sigma}
\left\{
x\in\mathbb R^k :
\operatorname{sign} P_i(a;x)=\sigma_i \text{ for all } i=1,\ldots,s
\right\}.
\]
The compatibility above implies that, for fixed \(\alpha\) and \(\Sigma\), all
sets \(T_a^\Sigma\) with \(a\in B_\alpha\) are semialgebraically homeomorphic.

There are only finitely many strata \(B_\alpha\), and only finitely many
subsets \(\Sigma\subset\{-1,0,+1\}^s\). Therefore the semialgebraic sets
defined by Boolean formulas in at most \(s\) polynomials of degree at most
\(d\) occur in only finitely many semialgebraic homeomorphism types.

Now restrict to those fibers which are closed and bounded. For every pair
\((\alpha,\Sigma)\) for which there exists at least one closed and bounded
fiber \(T_a^\Sigma\) with \(a\in B_\alpha\), choose one such fiber and denote it
by
\[
T_{\alpha,\Sigma}.
\]
There are only finitely many such representatives.

Each \(T_{\alpha,\Sigma}\) is a closed and bounded semialgebraic set. By the
semialgebraic triangulation theorem, it admits a finite semialgebraic
triangulation; see \cite[Chapter~5]{BPR06}. Choose one such triangulation
\[
|K_{\alpha,\Sigma}|\xrightarrow{\sim} T_{\alpha,\Sigma}.
\]
Define
\[
C=C(k,s,d):=
\max\left(
1,\,
\max_{\alpha,\Sigma}
\#\{\text{simplices of }K_{\alpha,\Sigma}\}
\right),
\]
where the maximum is taken over the finitely many chosen representatives.

We claim that every closed and bounded semialgebraic set \(T\subset\mathbb R^k\)
defined by at most \(s\) polynomials of degree at most \(d\) admits a finite
semialgebraic triangulation with at most \(C\) simplices. Indeed, such a set is
some \(T_a^\Sigma\), with \(a\in B_\alpha\). Since \(T_a^\Sigma\) is closed and
bounded, the pair \((\alpha,\Sigma)\) was included above. By the preceding
triviality statement, \(T_a^\Sigma\) is semialgebraically homeomorphic to the
chosen representative \(T_{\alpha,\Sigma}\). Pulling back the chosen
triangulation of \(T_{\alpha,\Sigma}\) gives a semialgebraic triangulation of
\(T_a^\Sigma\) with at most \(C\) simplices.

We now use this uniform triangulation bound to control torsion in cohomology.
Choose a triangulation
\[
|K|\xrightarrow{\sim} T
\]
with at most \(C\) simplices. Then
\[
H^q(T,\mathbb Z)\cong H^q(K,\mathbb Z),
\]
where the right-hand side denotes simplicial cohomology.

Choose orientations of all simplices of \(K\). The simplicial cochain complex is
\[
C^{q-1}(K,\mathbb Z)
\xrightarrow{\delta^{q-1}}
C^q(K,\mathbb Z)
\xrightarrow{\delta^q}
C^{q+1}(K,\mathbb Z).
\]
Each group \(C^i(K,\mathbb Z)\) is free abelian of rank at most \(C\). With
respect to the oriented simplex bases, the matrices of the coboundary maps have
entries in \(\{-1,0,1\}\).

Let \(A\) be the matrix of
\[
\delta^{q-1}:C^{q-1}(K,\mathbb Z)\to C^q(K,\mathbb Z).
\]
There is an exact sequence
\[
0
\longrightarrow
H^q(K,\mathbb Z)
\longrightarrow
\operatorname{coker}(\delta^{q-1})
\longrightarrow
\operatorname{im}(\delta^q)
\longrightarrow
0.
\]
Since \(\operatorname{im}(\delta^q)\) is a subgroup of a free abelian group, it
is free abelian. In particular, it is torsion-free. Therefore
\[
H^q(K,\mathbb Z)_{\mathrm{tors}}
\cong
\operatorname{coker}(A)_{\mathrm{tors}}.
\]

Write
\[
A:\mathbb Z^m\to \mathbb Z^n,
\qquad m,n\le C,
\]
and let
\[
r=\operatorname{rank}A.
\]
If \(r=0\), then \(\operatorname{coker}(A)\) is free abelian, so there is no
torsion. Thus assume \(r>0\).

By Smith normal form,
\[
\operatorname{coker}(A)
\cong
\mathbb Z^{n-r}
\oplus
\bigoplus_{j=1}^r \mathbb Z/d_j\mathbb Z,
\]
where
\[
d_j>0,\qquad d_j\mid d_{j+1}.
\]
Thus
\[
\#\operatorname{coker}(A)_{\mathrm{tors}}
=
d_1\cdots d_r,
\]
and
\[
\exp \operatorname{coker}(A)_{\mathrm{tors}}
=
d_r
\le d_1\cdots d_r.
\]

The product \(d_1\cdots d_r\) is the greatest common divisor of the nonzero
\(r\times r\) minors of \(A\). In particular, it divides every nonzero
\(r\times r\) minor of \(A\).

Choose a nonzero \(r\times r\) minor \(B\) of \(A\). All entries of \(B\) lie in
\(\{-1,0,1\}\). Therefore each row of \(B\) has Euclidean norm at most
\(\sqrt r\). By Hadamard's inequality,
\[
|\det B|
\le
(\sqrt r)^r
=
r^{r/2}.
\]
Since \(r\le C\), we have
\[
|\det B|\le C^C.
\]
Because \(d_1\cdots d_r\) divides \(\det B\), it follows that
\[
d_1\cdots d_r\le C^C.
\]
Consequently,
\[
\#H^q(K,\mathbb Z)_{\mathrm{tors}}
\le C^C,
\qquad
\exp H^q(K,\mathbb Z)_{\mathrm{tors}}
\le C^C.
\]
Using \(H^q(T,\mathbb Z)\cong H^q(K,\mathbb Z)\), the same bounds hold for
\(T\).

Thus the lemma holds with
\[
M:=C(k,s,d)^{C(k,s,d)}.
\]
This integer depends only on \(k,s\), and \(d\), and is independent of \(q\).
\end{proof}

\subsection{Local Picard groups and link torsion}\label{sec:local_picard_link}

We now record the local Picard input which converts torsion in the link into a bound for the residual Cartier index. The surface case uses Brieskorn's class group computation, and the higher-dimensional case uses Koll{\'a}r's exponential sequence argument.

\begin{lem}[Chern class injection for the punctured local germ]
\label{lem:chern_injection}
Let \(R\) be the local ring at a closed point \(x\) of a normal complex algebraic variety.
Assume that \(R\) has a rational singularity and that
\[
\dim R=n\ge 3.
\]
Let
\[
U:=\Spec R\setminus\{x\}.
\]
Let \(M\) be the topological link of the complex analytic germ \((x,\Spec R)\). Then the first Chern class map gives an injection
\[
c_1:\Pic(U)\hookrightarrow H^2(M,\bbZ).
\]
\end{lem}

\begin{proof}
Choose a sufficiently small contractible Stein analytic representative \(W\) of the analytic germ
\((x,\Spec R)\), and set
\[
U^{an}:=W\setminus\{x\}.
\]
Then \(U^{an}\) is homotopy equivalent to the topological link \(M\). Hence
\[
H^2(U^{an},\bbZ)\cong H^2(M,\bbZ).
\]

We recall Kollár's exponential sequence argument. By Kollár
\cite[\S4, (34), ``Exponential sequence'']{Kol16}, the analytic exponential sequence on
\(U^{an}\),
\[
0\longrightarrow \bbZ
\longrightarrow \sO_{U^{an}}
\xrightarrow{\exp(2\pi i\,\cdot)}
\sO_{U^{an}}^\times
\longrightarrow 0,
\]
gives the exact segment
\[
H^1(U^{an},\sO_{U^{an}})
\longrightarrow
\Pic(U^{an})
\xrightarrow{c_1}
H^2(U^{an},\bbZ).
\]
Thus the kernel of \(c_1\) is contained in the image of
\(H^1(U^{an},\sO_{U^{an}})\). Equivalently, in the notation of Kollár's
Proposition 35, this kernel is the subgroup denoted
\[
\Pic^{\mathrm{loc}-\circ}(x,X).
\]

We claim that
\[
H^1(U^{an},\sO_{U^{an}})=0.
\]
Indeed, the local cohomology exact sequence for the pair \((W,\{x\})\) contains
\[
H^1(W,\sO_W)
\longrightarrow
H^1(U^{an},\sO_{U^{an}})
\longrightarrow
H^2_x(W,\sO_W)
\longrightarrow
H^2(W,\sO_W).
\]
Since \(W\) is Stein, Cartan's Theorem B gives
\[
H^1(W,\sO_W)=H^2(W,\sO_W)=0.
\]
Therefore
\[
H^1(U^{an},\sO_{U^{an}})
\cong
H^2_x(W,\sO_W).
\]

Since \(R\) has a rational singularity, it is Cohen--Macaulay. Hence
\[
\depth_x R=\dim R=n\ge 3.
\]
By the depth criterion for local cohomology, for instance Hartshorne
\cite[Chapter~III, Theorem~3.8]{Har67}, we have
\[
H^i_x(W,\sO_W)=0
\qquad\text{for }i<3.
\]
In particular,
\[
H^2_x(W,\sO_W)=0.
\]
Thus
\[
H^1(U^{an},\sO_{U^{an}})=0.
\]

Consequently the first Chern class map is injective:
\[
c_1:\Pic(U^{an})\hookrightarrow H^2(U^{an},\bbZ)
\cong H^2(M,\bbZ).
\]
By \cite[Section 3.5, (5.1.1), (5.1.2)]{Kol16}, it induces an injective map from the algebraic Picard group, and by abusing notation, we still denote it by $c_1$:
\[
c_1: \Pic(U)\hookrightarrow H^2(U^{an},\bbZ)
\cong H^2(M,\bbZ).\]
The lemma follows.
\end{proof}

\begin{lem}[Residual Cartier index bounded by link torsion]
\label{lem:residual_index_link}
Let $R$ be the local ring at a closed point $x$ of a normal complex algebraic variety. Assume that
$R$ has a rational singularity and that
\[
\dim R=d\ge 2.
\]
Let
\[
U\coloneqq \Spec R\setminus\{x\}.
\]
Let $M$ be the topological link of the complex analytic germ $(x,\Spec R)$. Let
\[
e=\exp H^2(M,\bbZ)_{\mathrm{tors}},
\]
with $e=1$ if this torsion subgroup is zero. Let $n$ be a positive integer. Let $D$ be a Weil divisor on $\Spec R$ that is $\mathbb Q$-Cartier at $x$. Assume that $nD$ is Cartier at every point of $U$. Then the Cartier index of $D$ at $x$ divides $ne$.
\end{lem}

\begin{proof}
If $R$ is regular, then every Weil divisor on $\Spec R$ is Cartier, so there is nothing to prove. Thus we may assume, when needed, that the germ is singular.

The divisor $nD$ is Cartier on $U$, so it defines a line bundle
\[
L\coloneqq \sO_U(nD)\in \Pic(U).
\]
Since $D$ is $\mathbb Q$-Cartier at $x$, there exists $r>0$ such that $rD$ is Cartier on the local scheme $\Spec R$. Hence $rnD$ is Cartier on $\Spec R$. Since $\Spec R$ is local, every line bundle on $\Spec R$ is trivial; equivalently
\[
\sO_{\Spec R}(rnD)\cong \sO_{\Spec R}.
\]
Restricting to $U$, we obtain
\[
L^r=\sO_U(rnD)\cong \sO_U.
\]
Thus $L$ has finite order in $\Pic(U)$. Let
\[
q\coloneqq\ord_{\Pic(U)}(L).
\]

We first show that $q\mid e$. Suppose first that $d=2$. Let $R^{an}$ be the analytic local ring at $x$, and let $U^{an}= \Spec R^{an}\setminus\{x\}$. Since $R^{an}$ is normal and restriction identifies the Weil divisor class group of $U^{an}$ to the Weil divisor class group of $\Spec R^{an}$, we have an injection: 
\[
\Pic(U^{an})\hookrightarrow \Cl (U^{an})\cong\Cl(R^{an}),
\]
By \cite[Section 3.5, (5.1.1)]{Kol16}, we have an injection:
\[\Pic (U)\hookrightarrow \Pic(U^{an})\],
By \cite[Satz~1.5]{Bri68}
\[
\Cl(R^{an})\cong H_1(M,\bbZ),
\]
and this group is finite; more precisely, its order is the absolute value of the determinant of the exceptional intersection matrix. The link \(M\) is a closed oriented three-manifold. Poincar\'e duality identifies \(H^2(M,\bbZ)\) with \(H_1(M,\bbZ)\) up to the standard universal coefficient identification for a closed oriented three-manifold; in particular the torsion exponents agree:
\[
H^2(M,\bbZ)\cong H_1(M,\bbZ).
\]
Thus \(\Pic(U)\) is a finite group whose exponent divides \(e\). Therefore the order \(q\) of \(L\) divides \(e\).

Now suppose that $d\ge 3$. By Lemma~\ref{lem:chern_injection}, the first Chern class map gives an injection
\[
\Pic(U)\hookrightarrow H^2(M,\bbZ).
\]
Therefore the order $q$ of $L$ is the same as the order of its image in $H^2(M,\bbZ)$. Since $L$ is torsion, its image lies in $H^2(M,\bbZ)_{\mathrm{tors}}$. Hence $q\mid e$.

By definition of $q$, we have
\[
L^q\cong \sO_U.
\]
Equivalently,
\[
\sO_U(qnD)\cong \sO_U.
\]
Thus the Cartier divisor $qnD$ is principal on $U$. Hence there exists
\[
f\in \Frac(R)^\times
\]
such that
\[
(qnD)|_U=\divisor_U(f)
\]
as Weil divisors on $U$.

Since $R$ is normal and $d\ge 2$, the closed point $x$ is not a codimension one point. Hence the schemes $\Spec R$ and $U$ have the same codimension one points. Therefore equality of Weil divisors on $U$ implies equality of Weil divisors on $\Spec R$:
\[
qnD=\divisor_{\Spec R}(f).
\]
Thus $qnD$ is principal on $\Spec R$, in particular Cartier at $x$.

Consequently the Cartier index of $D$ at $x$ divides $qn$. Since $q\mid e$, it divides $ne$.
\end{proof}

\section{Bounded degree local residual estimates}\label{sec:local_estimates}

The goal of this section is to prove the bounded degree local estimate used in the induction. We first prove the closed point form over an algebraically closed field isomorphic to \(\bbC\). We then prove the corresponding statement for arbitrary prime localizations by a Noether normalization reduction to the closed point case, followed by flat divisorial base change and faithfully flat descent.

\begin{lem}[Bounded degree geometric closed point residual Cartier index bound]
\label{lem:bounded_geom_residual_index}
Let \(N,s,\delta\) be integers with
\[
N\ge 1,\qquad s\ge 1,\qquad \delta\ge 1.
\]
Then there exists a positive integer
\[
e_{\mathrm{bd}}=e_{\mathrm{bd}}(N,s,\delta),
\]
depending only on \(N,s,\delta\), with the following property.

Let \(\Omega\) be an algebraically closed field of characteristic \(0\), and let
\[
\sigma:\Omega\longrightarrow \mathbb C
\]
be a field isomorphism. Let \(t\) be an integer with
\[
1\le t\le s.
\]
Let
\[
F_1,\ldots,F_t\in \Omega[z_1,\ldots,z_N]
\]
be polynomials, each of total degree at most \(\delta\). Let
\[
a_1,\ldots,a_N\in \Omega
\]
satisfy
\[
F_i(a_1,\ldots,a_N)=0
\qquad
\text{for every }1\le i\le t.
\]

Set
\[
A:=\Omega[z_1,\ldots,z_N]/(F_1,\ldots,F_t).
\]
Let \(\mathfrak m\subset A\) be the maximal ideal generated by the images of
\[
z_1-a_1,\ldots,z_N-a_N.
\]
Put
\[
R:=A_{\mathfrak m},
\qquad
U:=\Spec R\setminus\{\mathfrak m\}.
\]
Assume that \(R\) is a normal local domain, that \(R\) has a rational singularity, and that
\[
\dim R\ge 2.
\]

Let \(D\) be a Weil divisor on \(\Spec R\), and let \(n\) be a positive integer. Assume that \(D\) is \(\mathbb Q\)-Cartier at the closed point of \(\Spec R\), and that \(nD\) is Cartier at every point of \(U\). Then the Cartier index of \(D\) at the closed point of \(\Spec R\) divides
\[
ne_{\mathrm{bd}}.
\]
\end{lem}

\begin{proof}
Apply Lemma~\ref{lem:semialg_torsion} with
\[
k=2N,\qquad
s_0=2s+1,\qquad
d_0=\max\{\delta,2\},
\qquad
q=2.
\]
It gives an integer
\[
M=M(N,s,\delta)\ge 1
\]
such that every closed bounded semialgebraic subset of
\[
\mathbb R^{2N}
\]
defined using at most \(2s+1\) real polynomials of degree at most
\[
\max\{\delta,2\}
\]
has \(H^2(-,\mathbb Z)\)-torsion exponent at most \(M\).

Define
\[
e_{\mathrm{bd}}:=\operatorname{lcm}(1,2,\ldots,M).
\]

For each \(i\), let
\[
G_i\in \mathbb C[z_1,\ldots,z_N]
\]
be the polynomial obtained from \(F_i\) by applying \(\sigma\) to all coefficients. For \(1\le j\le N\), put
\[
b_j:=\sigma(a_j).
\]
Set
\[
A_\sigma:=\mathbb C[z_1,\ldots,z_N]/(G_1,\ldots,G_t).
\]
Let \(\mathfrak m_\sigma\subset A_\sigma\) be the maximal ideal generated by the images of
\[
z_1-b_1,\ldots,z_N-b_N.
\]
Put
\[
T:=(A_\sigma)_{\mathfrak m_\sigma}.
\]

The coefficient-wise map induced by \(\sigma\), together with
\[
z_j\longmapsto z_j,
\]
gives a ring isomorphism
\[
A\cong A_\sigma
\]
carrying \(\mathfrak m\) to \(\mathfrak m_\sigma\). After localizing, we obtain a local ring isomorphism
\[
R\cong T.
\]
Therefore \(T\) is a normal local domain, has a rational singularity, and has dimension at least \(2\).

Let
\[
L:=\operatorname{link}(\mathfrak m_\sigma,\Spec A_\sigma)
\]
be the topological link of the corresponding complex analytic germ. For sufficiently small \(\epsilon>0\), the link \(L\) is homeomorphic to the intersection of the common zero set of
\[
G_1,\ldots,G_t
\]
inside \(\mathbb C^N\) with the Euclidean sphere
\[
\sum_{j=1}^{N}|z_j-b_j|^2=\epsilon^2.
\]
Thus \(L\) is a closed and bounded semialgebraic subset of
\[
\mathbb R^{2N}\cong \mathbb C^N.
\]
It is described by the real and imaginary parts of
\[
G_1,\ldots,G_t
\]
together with one quadratic sphere equation. Hence \(L\) is described using at most
\[
2t+1\le 2s+1
\]
real polynomials, each of degree at most
\[
\max\{\delta,2\}.
\]

By the choice of \(M\), the torsion subgroup of
\[
H^2(L,\mathbb Z)
\]
has exponent at most \(M\). Let
\[
e_{\mathrm{link}}:=\exp H^2(L,\mathbb Z)_{\mathrm{tors}},
\]
with the convention \(e_{\mathrm{link}}=1\) if the torsion subgroup is zero. Then
\[
e_{\mathrm{link}}\mid e_{\mathrm{bd}}.
\]

Transport \(D\) through the local ring isomorphism \(R\cong T\) to a Weil divisor \(D_\sigma\) on \(\Spec T\). The Cartier index of \(D\) at the closed point of \(\Spec R\) equals the Cartier index of \(D_\sigma\) at the closed point of \(\Spec T\), because the local ring isomorphism identifies the divisorial modules of all positive multiples of the two divisors. It also identifies the punctured spectra, so \(nD_\sigma\) is Cartier on the punctured spectrum of \(T\).

Lemma~\ref{lem:residual_index_link} applies to the normal complex rational local ring \(T\). It gives that the Cartier index of \(D_\sigma\) divides
\[
ne_{\mathrm{link}}.
\]
Since
\[
e_{\mathrm{link}}\mid e_{\mathrm{bd}},
\]
the Cartier index of \(D\) divides
\[
ne_{\mathrm{bd}}.
\]
This proves the lemma.
\end{proof}

\begin{prop}[Bounded degree affine prime residual Cartier index bound]\label{prop:prime_residual_bound}
Let \(N,t,\delta\) be integers with \(N\geq 1\), \(t\geq 1\), and \(\delta\geq 1\). Then there exists a positive integer \(e_{\mathrm{pr}}\), depending only on \(N,t,\delta\), with the following property.

Let \(P_1,\ldots,P_t\) be polynomials in \(\bbC[z_1,\ldots,z_N]\), each of total degree at most \(\delta\), and set
\[
A:=\bbC[z_1,\ldots,z_N]/(P_1,\ldots,P_t).
\]
Let \(\fp\) be a prime ideal of \(A\), and put \(R:=A_\fp\). Assume that \(R\) is a normal local domain with a rational singularity and that
\[
\dim R\ge 2.
\]
Let \(D\) be a Weil divisor on \(\Spec R\), and let \(n\) be a positive integer. Assume that \(D\) is \(\bbQ\)-Cartier at the closed point of \(\Spec R\), and that \(nD\) is Cartier on the punctured spectrum of \(R\). Then the Cartier index of \(D\) at the closed point of \(\Spec R\) divides
\[
ne_{\mathrm{pr}}.
\]
\end{prop}

\begin{proof}
Let
\[
e_{\mathrm{pr}}:=e_{\mathrm{bd}}(N,t+N,\delta),
\]
where \(e_{\mathrm{bd}}(N,t+N,\delta)\) is the integer supplied by Lemma~\ref{lem:bounded_geom_residual_index}. We prove that this integer works.

Let
\[
F:=\Frac(A/\fp),
\qquad
m:=\trdeg_\bbC F.
\]
Since \(A\) is generated by \(N\) elements over \(\bbC\), we have \(0\le m\le N\).

We use linear Noether normalization over the infinite field \(\bbC\): if \(k\) is an infinite field, \(V\) is an integral closed subvariety of affine \(N\)-space over \(k\), and \(m=\dim V\), then there are affine linear polynomials \(L_1,\ldots,L_m\) in the affine coordinates such that the coordinate ring \(k[V]\) is finite over the polynomial subring \(k[L_1,\ldots,L_m]\) \cite[Chapter~I, \S~1, Theorem~1]{Mum76}. Apply this to the integral affine variety \(\Spec(A/\fp)\) over \(\bbC\). Choose affine linear polynomials
\[
L_1,\ldots,L_m\in \bbC[z_1,\ldots,z_N]
\]
such that the images of these polynomials in \(A/\fp\) make \(A/\fp\) finite over \(\bbC[y_1,\ldots,y_m]\), where \(y_\ell\) is the image of \(L_\ell\). Let
\[
R_0:=\bbC[Y_1,\ldots,Y_m]
\]
with the convention \(R_0=\bbC\) if \(m=0\), and let \(R_0\to A\) send \(Y_\ell\) to the image of \(L_\ell\) in \(A\). The contraction of \(\fp\) to \(R_0\) is zero, and \(F\) is a finite field extension of
\[
K:=\Frac(R_0)=\bbC(Y_1,\ldots,Y_m).
\]
Since \(K\) has characteristic \(0\), the finite extension \(F/K\) is separable.

Let
\[
S_0:=R_0\setminus\{0\},
\qquad
B:=S_0^{-1}A,
\qquad
\fq_0:=S_0^{-1}\fp.
\]
Then \(\fq_0\) is a prime ideal of \(B\), and
\[
B/\fq_0=S_0^{-1}(A/\fp)
\]
is a finitely generated \(K\)-domain with fraction field \(F\). Since \(F/K\) is finite, \(B/\fq_0\) is a finite-dimensional \(K\)-domain, hence a field. Thus \(\fq_0\) is maximal and \(B/\fq_0=F\). By transitivity of localization,
\[
B_{\fq_0}\cong A_\fp=R.
\]
If \(\eta_\ell\) denotes the image of \(Y_\ell\) in \(K\), then
\[
B\cong K[z_1,\ldots,z_N]/(P_1,\ldots,P_t,L_1-\eta_1,\ldots,L_m-\eta_m).
\]
This presentation has \(t+m\le t+N\) equations, each of total degree at most \(\delta\).

Let \(J\) be the kernel of the localization homomorphism \(B\to B_{\fq_0}\). Since \(B_{\fq_0}\cong R\) is a domain, \(J\) is a prime ideal of \(B\). Put
\[
B':=B/J,
\qquad
\fq':=\fq_0/J.
\]
Then \(B'\) is a finitely generated \(K\)-domain, \((B')_{\fq'}\cong R\), and \(B'/\fq'=F\).

Let \(\Omega\) be an algebraic closure of \(F\), and let \(\iota:F\hookrightarrow\Omega\) be the inclusion. By Steinitz's theorem \cite{Ste10}, an algebraically closed field of characteristic \(0\) is determined up to abstract field isomorphism by its transcendence degree over \(\bbQ\). Here \(\Omega\) is algebraic over the finite extension \(F\) of \(\bbC(Y_1,\ldots,Y_m)\), and \(m\le N\). Hence
\[
\trdeg_\bbQ\Omega=\trdeg_\bbQ\bbC.
\]
Choose a field isomorphism
\[
\sigma:\Omega\longrightarrow\bbC.
\]

Let
\[
A_\Omega:=\Omega[z_1,\ldots,z_N]/(P_1,\ldots,P_t,L_1-\eta_1,\ldots,L_m-\eta_m),
\]
where the coefficients are viewed in \(\Omega\) through
\[
K\longrightarrow F\xrightarrow{\iota}\Omega.
\]
For \(1\le j\le N\), let \(a_j\) be the image of \(z_j\) in \(F=B/\fq_0\). Let \(\fm_\Omega\subset A_\Omega\) be the maximal ideal generated by the images of
\[
z_1-\iota(a_1),\ldots,z_N-\iota(a_N).
\]
The ring \((A_\Omega)_{\fm_\Omega}\) is the local ring at a closed point of a closed subscheme of \(\bbA^N_\Omega\) defined by \(t+m\le t+N\) equations of degree at most \(\delta\).

The quotient map \(B\to B'\) does not change this local ring after base change. Indeed, if \(j\in J\), then \(j/1=0\) in \(B_{\fq_0}\), so \(bj=0\) in \(B\) for some \(b\in B\setminus\fq_0\). The image of \(b\) is not in \(\fm_\Omega\), and therefore \(j\) becomes zero after localizing \(\Omega\otimes_KB\) at \(\fm_\Omega\). Hence
\[
(A_\Omega)_{\fm_\Omega}
\cong
(\Omega\otimes_KB')_{\fm_\Omega'},
\]
where \(\fm_\Omega'\) is the corresponding maximal ideal of \(\Omega\otimes_KB'\). Define
\[
S:=(\Omega\otimes_KB')_{\fm_\Omega'}.
\]
The natural local homomorphism
\[
R\cong (B')_{\fq'}\longrightarrow S
\]
is flat, because it is obtained from the field extension \(K\to\Omega\) by base change and localization.

The closed fiber is the localization of
\[
\Omega\otimes_KF
\]
at the maximal ideal corresponding to the embedding \(\iota:F\hookrightarrow\Omega\). Since \(F/K\) is finite separable and \(\Omega/F\) is algebraic, this localization is the field \(\Omega\). Thus \(R\to S\) is faithfully flat and local. In particular, it is injective.

Lemma~\ref{lem:alg_field_rational}, applied to the finitely generated \(K\)-domain \(B'\) and to the algebraic field extension \(\Omega/K\), shows that \(S\) is a normal local domain with a rational singularity, and that
\[
\dim S=\dim R\ge 2.
\]

Since \(R\to S\) is injective and flat and both rings are normal noetherian domains, Lemma~\ref{lem:generalized_flat_pullback} gives the flat divisorial base change \(D_S\) of \(D\) to \(\Spec S\). For every integer \(q\ge 1\), the base change of \(qD\) is \(qD_S\), and
\[
S\otimes_R\sO_R(qD)\cong \sO_S(qD_S).
\]
Since \(D\) is \(\bbQ\)-Cartier at the closed point of \(\Spec R\), this compatibility also shows that \(D_S\) is \(\bbQ\)-Cartier at the closed point of \(\Spec S\).

We next show that \(nD_S\) is Cartier on the punctured spectrum of \(S\). Let \(Q\) be a nonmaximal prime of \(S\), and put \(P:=Q\cap R\). If \(P\) were the maximal ideal of \(R\), then \(Q\) would define a prime of the closed fiber of \(R\to S\). Since this closed fiber is a field and \(S\) is local, this would force \(Q\) to be the maximal ideal of \(S\), a contradiction. Hence \(P\) is nonmaximal. Since \(nD\) is Cartier on the punctured spectrum of \(R\), the module \(\sO_R(nD)_P\) is a free \(R_P\)-module of rank one. Localizing the divisorial compatibility at \(Q\), we get
\[
\sO_S(nD_S)_Q
\cong
S_Q\otimes_{R_P}\sO_R(nD)_P,
\]
which is a free \(S_Q\)-module of rank one. Thus \(nD_S\) is Cartier at \(Q\). Since \(Q\) was arbitrary, \(nD_S\) is Cartier on the punctured spectrum of \(S\).

We now apply Lemma~\ref{lem:bounded_geom_residual_index} to the local ring
\[
S\cong (A_\Omega)_{\fm_\Omega}
\]
with the field isomorphism \(\sigma:\Omega\to\bbC\), the divisor \(D_S\), and the integer \(n\). The hypotheses have been verified above, and the bounded presentation uses at most \(t+N\) equations of degree at most \(\delta\). Therefore the Cartier index of \(D_S\) at the closed point of \(\Spec S\) divides
\[
ne_{\mathrm{pr}}.
\]
Equivalently, \(ne_{\mathrm{pr}}D_S\) is Cartier on \(\Spec S\).

Using the divisorial compatibility with \(q=ne_{\mathrm{pr}}\), we obtain
\[
S\otimes_R\sO_R(ne_{\mathrm{pr}}D)
\cong
\sO_S(ne_{\mathrm{pr}}D_S).
\]
The right-hand side is a free \(S\)-module of rank one. Since \(R\to S\) is faithfully flat and local, Lemma~\ref{lem:ff_detects_cartier} implies that \(ne_{\mathrm{pr}}D\) is Cartier on \(\Spec R\). Thus the Cartier index of \(D\) divides \(ne_{\mathrm{pr}}\), as required.
\end{proof}

\section{The universal local theorem}\label{sec:universal_local}

In this section, we prove the bounded degree local theorem in all dimensions by induction on the local dimension.

\begin{thm}[Universal local Cartier index bound]\label{thm:universal_local}
Let \(N\), \(t\), and \(\delta\) be integers with \(N\geq 1\), \(t\geq 1\), and \(\delta\geq 1\). Then there exists a positive integer \(C\), depending only on \(N\), \(t\), and \(\delta\), with the following property.

Let \(P_1,\ldots,P_t\) be polynomials in \(\bbC[z_1,\ldots,z_N]\), each of total degree at most \(\delta\), and set
\[
A\coloneqq \bbC[z_1,\ldots,z_N]/(P_1,\ldots,P_t).
\]
Let \(\fp\) be a prime ideal of \(A\), and put \(R\coloneqq A_\fp\). Assume that \(R\) is a normal local domain and that \(R\) has a rational singularity. Let \(D\) be a Weil divisor on \(\Spec R\) that is \(\bbQ\)-Cartier at the closed point of \(\Spec R\). Then the Cartier index of \(D\) at the closed point of \(\Spec R\) divides \(C\).
\end{thm}

\begin{proof}
Fix \(N\geq 1\), \(t\geq 1\), and \(\delta\geq 1\). We construct positive integers \(C_n(N,t,\delta)\), for \(0\leq n\leq N\), such that \(C_n(N,t,\delta)\) works for all local rings in the statement of dimension at most \(n\).

Set
\[
C_0(N,t,\delta)=C_1(N,t,\delta)=1.
\]
A normal local domain of dimension \(0\) is a field, and a normal local domain of dimension \(1\) is a discrete valuation ring. In both cases every Weil divisor is Cartier.

Let \(e_{\mathrm{pr}}(N,t,\delta)\) be the integer supplied by Proposition~\ref{prop:prime_residual_bound}. For \(2\leq n\leq N\), assuming that \(C_{n-1}(N,t,\delta)\) has been constructed, put
\[
I_{\mathrm{low}}:=C_{n-1}(N,t,\delta),
\]
and define
\[
C_n(N,t,\delta):=I_{\mathrm{low}}e_{\mathrm{pr}}(N,t,\delta).
\]
We prove by induction that these integers have the required property. The cases of dimension at most \(1\) were just discussed. Assume now that \(n\ge 2\), and that the assertion is known in dimensions less than \(n\). The cases of dimension less than \(n\) are still covered by \(C_n(N,t,\delta)\), because \(C_{n-1}(N,t,\delta)\) divides \(C_n(N,t,\delta)\).

Let \(R=A_\fp\) be a local ring in the statement with
\[
\dim R=n,
\]
and let \(D\) be a Weil divisor on \(\Spec R\) that is \(\bbQ\)-Cartier at the closed point. We first prove that
\[
I_{\mathrm{low}}D
\]
is Cartier on the punctured spectrum of \(R\). Let \(P\) be a nonmaximal prime of \(R\). Then \(R_P\) is a normal local domain. The localization also has a rational singularity: choosing a resolution of \(\Spec R\) with the vanishing of higher direct images of the structure sheaf, its restriction over \(\Spec R_P\) remains proper and birational, and the higher direct image vanishing localizes. Moreover \(\dim R_P<n\), and \(R_P\) is a prime localization of the original affine \(\bbC\)-algebra \(A\). The induction hypothesis therefore shows that the Cartier index of \(D\) at \(P\) divides \(I_{\mathrm{low}}\). Hence \(I_{\mathrm{low}}D\) is Cartier at \(P\). Since \(P\) was arbitrary, \(I_{\mathrm{low}}D\) is Cartier on the punctured spectrum of \(R\).

We may now apply Proposition~\ref{prop:prime_residual_bound} with the auxiliary integer \(I_{\mathrm{low}}\). It gives that the Cartier index of \(D\) at the closed point of \(\Spec R\) divides
\[
I_{\mathrm{low}}e_{\mathrm{pr}}(N,t,\delta)=C_n(N,t,\delta).
\]
This proves the induction step.

Since every prime localization of a quotient of \(\bbC[z_1,\ldots,z_N]\) has dimension at most \(N\), the integer
\[
C:=C_N(N,t,\delta)
\]
proves the theorem.
\end{proof}

\section{Globalization, field descent, and proof of the main theorem}\label{sec:layer012}

In this section, we pass from the universal local theorem to the global bounded family statement.

\begin{lem}[Closed point reduction for $\tci$]\label{lem:closed_point_tci}
Let $k$ be an algebraically closed field. Let $B$ be a finite type $k$-scheme. Let $\pi\colon X\to B$ be a projective morphism. Let $I$ be a positive integer. Assume that for every closed point $b$ of $B$, every normal projective pure-dimensional fibre $Y=X_b$, every closed point $y$ of $Y$, and every Weil $\bbQ$-Cartier divisor $D$ on $Y$, the Cartier index of the localized divisor $D_y$ on $\Spec(\sO_{Y,y})$ divides $I$.

Then, for every closed point $b$ of $B$ such that $Y=X_b$ is a normal projective pure-dimensional fibre, the total Cartier index $\tci(Y)$ divides $I$.
\end{lem}

\begin{proof}
Fix a closed point $b$ of $B$ such that $Y=X_b$ is normal, projective, pure-dimensional. Let $D$ be a Weil $\bbQ$-Cartier divisor on $Y$. The hypothesis says that, for every closed point $y$ of $Y$, the divisor $ID_y$ is Cartier on $\Spec(\sO_{Y,y})$. Equivalently, the divisorial $\sO_Y$-module $\sO_Y(ID)$ is locally free of rank one at every closed point of $Y$.

Let $V$ be the locally free locus of the coherent $\sO_Y$-module $\sO_Y(ID)$. The set $V$ is open in $Y$, and contains every closed point of $Y$. Since $Y$ is of finite type over the algebraically closed field $k$, the scheme $Y$ is Jacobson; every nonempty closed subset of $Y\setminus V$ would contain a closed point of $Y$. The complement $Y\setminus V$ has no closed point, so $Y\setminus V$ is empty. Thus $\sO_Y(ID)$ is locally free of rank one on all of $Y$, and $ID$ is Cartier on $Y$.

The argument applies to every Weil $\bbQ$-Cartier divisor $D$ on $Y$. Therefore the Cartier index of every such $D$ divides $I$, so $\tci(Y)$ divides $I$.
\end{proof}

\begin{thm}[Complex projective families]\label{thm:globalization_C}
Let $B$ be a finite type complex scheme, and let $\pi\colon X\to B$ be a projective morphism.

Then there exists a positive integer $I$ with the following property. For every closed point $b$ of $B$, if the fibre $Y=X_b$ is a normal projective pure-dimensional complex variety and every local ring of $Y$ has a rational singularity, then $\tci(Y)$ divides $I$.
\end{thm}

\begin{proof}
Since $\pi$ is projective, there is an integer $N_0\geq 0$ and a closed immersion
\[
i_0\colon X\hookrightarrow B\times_{\bbC}\bbP^{N_0}_{\bbC}
\]
over $B$. If $N_0=0$, we compose $i_0$ with the standard closed immersion $\bbP^0_{\bbC}\hookrightarrow\bbP^1_{\bbC}$. Thus, after replacing $N_0$ by some $N\geq 1$, we may and do fix a closed immersion
\[
i\colon X\hookrightarrow B\times_{\bbC}\bbP^N_{\bbC}
\]
over $B$.

Choose a finite affine open cover \(B_\alpha = \Spec(S_\alpha)\) of \(B\). For each standard affine chart \(U_j\) of \(\bbP^N_\bbC\), the closed subscheme \(X\cap (B_\alpha\times U_j)\) is cut out inside \(B_\alpha\times \bbA^N_\bbC\) by finitely many polynomials in the affine coordinates on \(U_j\) with coefficients in \(S_\alpha\). After padding by the zero polynomial if necessary, let \(t_{\alpha,j}\ge 1\) be the number of such equations and let \(\delta_{\alpha,j}\) be an upper bound for their total degrees in the affine fiber coordinates. There are only finitely many pairs \((\alpha,j)\).

For each pair \((\alpha,j)\), Theorem~\ref{thm:universal_local} gives a positive integer \(C_{\alpha,j}\), depending only on \(N\), \(t_{\alpha,j}\), and \(\delta_{\alpha,j}\), with the following property: every normal rational prime localization of every complex affine \(N\)-space subscheme defined by at most \(t_{\alpha,j}\) equations of degree at most \(\delta_{\alpha,j}\) has all closed point local \(\bbQ\)-Cartier Weil divisor indices dividing \(C_{\alpha,j}\).

Define
\[
I_0\coloneqq \lcm_{\alpha,j} C_{\alpha,j}.
\]

We verify the local hypothesis of Lemma~\ref{lem:closed_point_tci}. Let \(b\) be a closed point of \(B\), let \(Y=X_b\) be a normal projective pure-dimensional fiber with rational singularities at all local rings, let \(y\) be a closed point of \(Y\), and let \(D\) be a Weil \(\bbQ\)-Cartier divisor on \(Y\). Choose \(\alpha\) with \(b\in B_\alpha\) and choose a standard affine chart \(U_j\) of \(\bbP^N_\bbC\) containing the image of \(y\). After specializing the equations for \(X\cap (B_\alpha\times U_j)\) at \(b\), the affine chart \(Y\cap U_j\) is a closed subscheme of \(\bbA^N_\bbC\) defined by at most \(t_{\alpha,j}\) complex polynomials of degree at most \(\delta_{\alpha,j}\). The local ring \(\sO_{Y,y}\) is the localization at the prime corresponding to \(y\) of the coordinate ring of this affine chart.

The localized divisor \(D_y\) on \(\Spec(\sO_{Y,y})\) is \(\bbQ\)-Cartier at the closed point. The local ring \(\sO_{Y,y}\) is a normal local domain with rational singularity by hypothesis on \(Y\). Theorem~\ref{thm:universal_local} therefore shows that the Cartier index of \(D_y\) divides \(C_{\alpha,j}\), hence divides \(I_0\).

Lemma~\ref{lem:closed_point_tci} applies with \(I=I_0\), and gives \(\tci(Y)\mid I_0\) for every fiber \(Y\) in the statement.
\end{proof}

\subsection{Proof of the main theorem}\label{sec:main_proof}

\begin{thm}[$=$Theorem \ref{thm:main}]\label{thm:projective_descent}
Let $k$ be an algebraically closed field of characteristic $0$. Let $B$ be a finite type $k$-scheme, and let $\pi\colon X\to B$ be a projective morphism. Then there exists a positive integer $I$ with the following property. For every closed point $b$ of $B$, if the fibre $Y=X_b$ is a normal projective pure-dimensional $k$-variety and every local ring of $Y$ has a rational singularity, then $\tci(Y)$ divides $I$.
\end{thm}

\begin{proof}
Choose a finitely generated subfield $k_0\subset k$ and a projective model
\[
\pi_0\colon X_0\to B_0
\]
over $k_0$ whose base change to $k$ is $\pi\colon X\to B$. Since $\pi_0$ is projective, after enlarging $k_0$ if necessary we may also choose an integer $N_0\geq 0$ and a closed immersion
\[
i_0\colon X_0\hookrightarrow B_0\times_{k_0}\bbP^{N_0}_{k_0}
\]
over $B_0$. If $N_0=0$, compose with $\bbP^0_{k_0}\hookrightarrow\bbP^1_{k_0}$, so that the model is embedded in some $B_0\times_{k_0}\bbP^N_{k_0}$ with $N\geq 1$. This embedding is only used to make the bounded degree reduction in the complex model; the statement itself depends only on the projective morphism.

Choose an embedding $\tau_0\colon k_0\hookrightarrow\bbC$, and let
\[
\pi_\bbC\colon X_\bbC\to B_\bbC
\]
be the base change of $\pi_0$ along $\tau_0$. Let $I$ be the integer supplied by Theorem~\ref{thm:globalization_C} for the complex projective family $\pi_\bbC$.

We prove that this \(I\) works over \(k\). Let \(b\) be a closed point of \(B\), put \(Y\coloneqq X_b\), and assume that \(Y\) is normal, projective, pure-dimensional, and has rational singularities at all local rings. Let \(D\) be a Weil \(\bbQ\)-Cartier divisor on \(Y\). It is enough to prove that \(ID\) is Cartier.

The point \(b\), the finitely many prime divisors in the support of \(D\), their coefficients, and the finite presentation data expressing that \(D\) is \(\bbQ\)-Cartier are defined over a finitely generated extension field of \(k_0\) inside \(k\). More concretely, choose affine charts of \(B_0\), \(X_0\), and \(Y\), choose equations for the closed point \(b\) in those charts, choose equations for the irreducible codimension one components in the support of \(D\), and choose local generators witnessing that some positive multiple of \(D\) is Cartier. All these data involve only finitely many coefficients in \(k\). After adjoining those coefficients to \(k_0\), we obtain a finitely generated field \(k_1\) with \(k_0\subset k_1\subset k\), a \(k_1\)-point \(b_1\) of \(B_0\times_{k_0}k_1\), a fiber \(Y_1\), and a Weil \(\bbQ\)-Cartier divisor \(D_1\) on \(Y_1\) whose base change to \(k\) is \((Y,D)\).

We also need \(Y_1\) to have the same singularity properties as \(Y\). Normality, projectivity, and pure dimensionality are open or constructible conditions in the finite presentation family after the above choices. For rational singularities, choose affine charts and a resolution with the vanishing of higher direct images after base change to \(k\); the resolution, the relevant coherent sheaves, and the vanishing of the finitely many cohomology groups involved are defined after adjoining finitely many further coefficients. Enlarging \(k_1\) once more, we may therefore arrange that \(Y_1\) is normal, projective, pure-dimensional, and has rational singularities at all local rings.

Choose an embedding $\tau_1\colon k_1\hookrightarrow\bbC$ extending $\tau_0$, and set
\[
Y_\bbC\coloneqq Y_1\times_{k_1,\tau_1}\bbC.
\]
Then \(Y_\bbC\) is a closed fiber of the complex projective family \(X_\bbC\to B_\bbC\). Projectivity and pure dimension are preserved by field extension. For every point \(y_\bbC\in Y_\bbC\), with image \(y_1\in Y_1\), Lemma~\ref{lem:alg_field_rational} applied on an affine neighbourhood of \(y_1\) and to the field extension \(k_1\subset\bbC\) shows that \(\sO_{Y_\bbC,y_\bbC}\) is a normal local domain with a rational singularity. Hence \(Y_\bbC\) satisfies the hypotheses of Theorem~\ref{thm:globalization_C}.

We now base change the divisor. On normal affine open subsets, and componentwise if the base change is disconnected, apply Lemma~\ref{lem:generalized_flat_pullback} to the induced injective flat homomorphisms of normal domains coming from \(k_1\subset\bbC\). These local constructions agree on overlaps, because on every common normal affine open subset the coefficient of a height one prime is computed by the same valuation formula of Lemma~\ref{lem:generalized_flat_pullback}. Thus they glue to a Weil divisor \(D_\bbC\) on \(Y_\bbC\), namely the flat divisorial base change of \(D_1\), such that for every integer \(q\geq 1\),
\[
\sO_{Y_\bbC}(qD_\bbC)
\cong
\bbC\otimes_{k_1}\sO_{Y_1}(qD_1).
\]
If \(rD_1\) is Cartier, then the displayed isomorphism with \(q=r\) shows that \(rD_\bbC\) is Cartier. Hence \(D_\bbC\) is \(\bbQ\)-Cartier.

By Theorem~\ref{thm:globalization_C}, the Cartier index of every Weil \(\bbQ\)-Cartier divisor on \(Y_\bbC\) divides \(I\). Applying this to \(D_\bbC\), we get that \(ID_\bbC\) is Cartier on \(Y_\bbC\).

We descend this Cartierness to \(Y_1\). Let \(y_1\in Y_1\) be any point. Since \(k_1\to\bbC\) is faithfully flat, there is a point \(y_\bbC\in Y_\bbC\) lying over \(y_1\). The local homomorphism
\[
\sO_{Y_1,y_1}\longrightarrow \sO_{Y_\bbC,y_\bbC}
\]
is flat and local with nonzero closed fiber, hence faithfully flat. Localizing the displayed divisorial module compatibility at \(y_\bbC\) and taking \(q=I\) gives
\[
\sO_{Y_\bbC,y_\bbC}(ID_\bbC)
\cong
\sO_{Y_\bbC,y_\bbC}\otimes_{\sO_{Y_1,y_1}}\sO_{Y_1,y_1}(ID_1).
\]
The left-hand side is a free module of rank one. By Lemma~\ref{lem:ff_detects_cartier}, \(\sO_{Y_1,y_1}(ID_1)\) is a free \(\sO_{Y_1,y_1}\)-module of rank one. Since \(y_1\) was arbitrary, \(ID_1\) is Cartier on \(Y_1\).

Finally, base change from \(k_1\) to \(k\) carries \(\sO_{Y_1}(ID_1)\) to \(\sO_Y(ID)\), so \(ID\) is Cartier on \(Y\). Since \(D\) was arbitrary, every Cartier index entering \(\tci(Y)\) divides \(I\). Thus \(\tci(Y)\) divides \(I\).
\end{proof}

\end{document}